\documentclass[12pt, reqno]{amsart}
\usepackage{amssymb,amsthm,amsmath}
\usepackage[numbers,sort&compress]{natbib}
\usepackage{amssymb,amsmath}
\usepackage{amsfonts}
\usepackage{mathrsfs}
\usepackage{latexsym}
\usepackage{amssymb}
\usepackage{amsthm}
\usepackage{indentfirst}
\usepackage{colortbl}
\date{\today}

\let\oldsection\section
\renewcommand\section{\setcounter{equation}{0}\oldsection}

\newtheorem{corollary}{Corollary}[section]
\newtheorem{theorem}{Theorem}[section]

\newtheorem{proposition}{Proposition}[section]

\newtheorem{remark}{Remark}[section]

\allowdisplaybreaks
\begin{document}

\title[Global solutions of compressible Navier-Stokes equations with vacuum]{Global small solutions of heat conductive compressible Navier-Stokes equations with vacuum: smallness on scaling invariant quantity}



\author{Jinkai~Li}
\address[Jinkai~Li]{South China Research Center for Applied Mathematics and
                Interdisciplinary Studies, South China Normal University,
                Zhong Shan Avenue West 55, Tianhe District, Guangzhou
                510631, China}
\email{jklimath@m.scnu.edu.cn; jklimath@gmail.com}





\begin{abstract}
In this paper, we consider the Cauchy problem to the heat conductive compressible Navier-Stokes equations in the presence of vacuum and with vacuum far field.
Global well-posedness of strong solutions
is established under the assumption, among some other regularity and compatibility  conditions, that the scaling invariant quantity $\|\rho_0\|_\infty(\|\rho_0\|_3+\|\rho_0\|_\infty^2\|\sqrt{\rho_0}u_0\|_2^2)(\|\nabla u_0\|_2^2+
  \|\rho_0\|_\infty\|\sqrt{\rho_0}E_0\|_2^2)$ is sufficiently small, with the smallness depending only on the parameters $R, \gamma, \mu, \lambda,$ and $\kappa$ in the system. The total mass can be either finite or infinite.
\end{abstract}

\maketitle

\allowdisplaybreaks
\section{Introduction}
In this paper, we consider the following heat conductive compressible Navier-Stokes equations for the ideal gas:
\begin{eqnarray}
  \partial_t\rho+\text{div}\,(\rho u)=0, \label{EQrho}\\
  \rho(\partial_tu+(u\cdot\nabla)u)-\mu\Delta u-(\mu+\lambda)\nabla\text{div}\, u+\nabla p=0,\label{EQu}\\
  c_v\rho(\partial_t\theta+u\cdot\nabla\theta)+p\text{div}\,u-\kappa\Delta\theta=\mathcal Q(\nabla u),\label{EQtheta}
\end{eqnarray}
in $\mathbb R^3\times(0,\infty)$, where the unknowns $\rho\geq0, u\in\mathbb R^3,$ and $\theta\geq0$, respectively, represent the density, velocity, and absolute temperature, $p=R\rho\theta$, with positive constant $R$, is the pressure, $c_v>0$ is a constant,
constants $\mu$ and $\lambda$ are the bulk and shear viscous coefficients, respectively, positive constant
$\kappa$ is the heat conductive coefficient, and
$$\mathcal Q(\nabla u)=\frac\mu2|\nabla u+(\nabla u)^t|^2+\lambda(\text{div}\,u)^2,$$
with $(\nabla u)^t$ being the transpose of $\nabla u$. The viscous coefficients $\mu$ and $\lambda$ satisfy the physical constraints
$$\mu>0,\quad 2\mu+3\lambda\geq0.$$
The additional assumption $2\mu>\lambda$ will also be use in this paper.

Due to their fundamental importance in the fluid dynamics, extensive studies have been carried out and many developments have been achieved
on the compressible Navier-Stokes equations in the last seventy years. The mathematical studies on the compressible Navier-Stokes equations
started with the uniqueness results
by Graffi \cite{GRAFFI53} in 1953 for barotropic fluid and by Serrin \cite{SERRIN59} in 1959 for general fluids,
and the local existence result by Nash \cite{NASH62} in 1962 for the Cauchy problem. Since then, comprehensive mathematical
theories have been developed for the compressible Navier-Stokes equations.

The mathematical theory for the compressible Navier-Stokes equations in 1D is satisfied and, in particular,
the corresponding global well-posedness, for arbitrary large initial data, and the
initial density can either be uniformly positive or only nonnegative (that is, it can vanish on some subset of the domain).
For the case that the initial density is uniformly positive, the global well-posedness of strong solutions, with large initial data,
was first proved in \cite{KANEL68}, for the isentropic case, and later in \cite{KAZHIKOV77}, for the general case,
and the corresponding large time behavior was recently proved in \cite{LILIANG16}, see also \cite{KAZHIKOV82,ZLOAMO97,ZLOAMO98,CHEHOFTRI00,JIAZLO04} for some
related results. For the case that the initial density contains vacuum, the corresponding global well-posedness of strong solutions
was recently proved by the author and his collaborator, see \cite{LJK18,LIXIN17,LIXIN19}.

Compared with the one dimensional case, the mathematical theory for the multi-dimensional case is far from satisfied and, in particular, some basic problems such
as the global existence of strong solutions and uniqueness of weak solution are still unknown.
For the case that the initial density is uniformly positive, the local well-posedness was proved long time ago,
see \cite{NASH62,ITAYA71,VOLHUD72,TANI77,VALLI82,LUKAS84} and, in particular, the inflow and outflow were allowed in \cite{LUKAS84}; however, the
general global well-posedness is still unknown. Global well-posedness of strong solutions with small initial data was first proved in
\cite{MATNIS80,MATNIS81,MATNIS82,MATNIS83}, and later further developed in many papers, see, e.g., \cite{PONCE85,VALZAJ86,DECK92,HOFF97,KOBSHI99,DANCHI01,CHENMIAOZHANG10,CHIDAN15,DANXU18,FZZ18}.
For the case that the initial density allows vacuum, global existence of weak solutions was first proved
in \cite{LIONS93,LIONS98}, see \cite{FEIREISL01,JIAZHA03,FEIREISL04P,FEIREISL04B,BRESCH18} for further developments, but the
uniqueness is still an open problem.
Local well-posedness of strong solutions was proved in \cite{CHOKIM04,CHOKIM06-1,CHOKIM06-2}, and the global well-posedness, with small initial data,
was proved in \cite{HLX12}, and see \cite{LIXIN13,HUANGLI11,WENZHU17} for further developments.

The aim of this paper is to establish the global existence of strong solutions to the Cauchy problem of (\ref{EQrho})--(\ref{EQtheta}), under some smallness assumptions on the initial data,
in the presence of initial vacuum, and with vacuum far field. The main novelty of this paper is that
the smallness assumption is imposed on some quantities that are
scaling invariant with respect to the following scaling transformation:
\begin{equation}
\label{SCAL}
(\rho_{0\lambda}(x), u_{0\lambda}(x), \theta_{0\lambda}(x))=(\rho_0(\lambda x), \lambda u_0(\lambda x), \lambda^2\theta_0(\lambda x)), \quad\forall\lambda\not=0.
\end{equation}
This scaling transformation on the initial data inheres in the following natural scaling invariant property
of system (\ref{EQrho})--(\ref{EQtheta}):
$$
\rho_\lambda(x,t)=\rho(\lambda x,\lambda^2t),\quad u_\lambda(x,t)=\lambda u(\lambda x,\lambda^2t),\quad\theta_\lambda(x,t)=\lambda^2\theta(\lambda x,\lambda^2t),
$$
that is, if $(\rho, u, \theta)$ is a solution, with initial data
$(\rho_0, u_0, \theta_0)$, then $(\rho_\lambda, u_\lambda, \theta_\lambda)$
is also a solution, for any nonzero $\lambda$, but with initial data $(\rho_{0\lambda}, u_{0\lambda}, \theta_{0\lambda})$.

The reason for us to focus on the
smallness assumptions on the scaling invariant quantities, rather than on those not, is the following fact: if assuming that $\mathscr M$ is a functional, satisfying
$$
\mathscr M(\rho_{0\lambda}, u_{0\lambda}, \theta_{0\lambda})=\lambda^\ell (\rho_0, u_0, \theta_0), \quad\forall\lambda\not=0,\quad\mbox{for some constant }\ell\not=0,
$$
and that the global well-posedness holds, for any initial data
$(\rho_0, u_0, \theta_0)$, such that $\mathscr M(\rho_0, u_0, \theta_0)\leq\varepsilon_0$, for some $\varepsilon_0>0$
depending only on the parameters of the system, then, by suitably choosing the scaling parameter $\lambda$, one can show that the system is actually globally well-posed, for arbitrary large initial data; however, this global well-posedness for
arbitrary large initial data is far from what we already known.

Before stating the main results, we first clarify some necessary notations
being used throughout this paper.
For $1\leq q\leq\infty$ and positive integer $m$, we use $L^q=L^q(\mathbb R^3)$ and $W^{1,q}=W^{m,q}(\mathbb R^3)$ to denote the standard Lebesgue and Sobolev spaces, respectively, and in the case that $q=2$, we use $H^m$
instead of $W^{m,2}$. For simplicity, we also use notations $L^q$ and
$H^m$ to denote the $N$ product spaces $(L^q)^N$ and $(H^m)^N$, respectively.
We always use $\|u\|_q$ to denote the $L^q$ norm of $u$. For shortening the
expressions, we sometimes use $\|(f_1,f_2,\cdots,f_n)\|_X$ to denote the sum
$\sum_{i=1}^N\|f_i\|_X$ or its equivalent norm $\left(\sum_{i=1}^N\|f_i\|_X^2
\right)^{\frac12}$. We denote
$$D^{k,r}=\Big\{ u\in L^1_{loc}(\mathbb{R}^3)\,\Big|\,\|\nabla^k u\|_r<\infty\Big\}, \quad D^k=D^{k,2},$$
$$D^{1}_0=\Big\{ u\in L^6\,\Big|\,\|\nabla u\|_2<\infty \Big\}.$$
For simplicity of notations, we adopt the notation
$$
\int fdx=\int_{\mathbb R^3}fdx.
$$

We are now ready to state the main result of this paper.

\begin{theorem}
\label{thm}
Assume $2\mu>\lambda$ and let $q\in(3,6]$ be a fixed constant. Assume that the initial data $(\rho_0, u_0, \theta_0)$ satisfies
\begin{eqnarray*}
  \rho_0,\theta_0\geq0,\quad\rho\leq\bar\rho,\quad \rho_0\in H^1\cap W^{1,q}, \quad\sqrt{\rho_0}\theta_0\in L^2, \quad (u_0, \theta_0)\in D_0^1\cap D^2, \\
  -\mu\Delta u_0-(\mu+\lambda)\nabla\text{div}\,u_0+\nabla p_0=\sqrt{\rho_0}g_1, \quad\kappa\Delta\theta_0+\mathcal Q(\nabla u)=\sqrt{\rho_0}g_2,
\end{eqnarray*}
for a positive constant $\bar\rho$ and some $(g_1, g_2)\in L^2$, where $p_0=R\rho_0\theta_0$.

Then, there is a positive number $\varepsilon_0$ depending only on $R,\gamma,\mu,\lambda,$ and $\kappa$, such that system (\ref{EQrho})--(\ref{EQtheta}), with initial data $(\rho_0, u_0, \theta_0)$, has a unique global solution $(\rho, u, \theta)$, satisfying
\begin{eqnarray*}
  &\rho\in C([0,\infty); H^1\cap W^{1,q}),\quad\rho_t\in C([0,\infty) L^2\cap L^q),\\
  &(u,\theta)\in C([0,\infty); D_0^1\cap D^2)\cap L^2_{\text{loc}}([0,\infty); D^{2,q}),\\
  &(u_t,\theta_t)\in L^2_{\text{loc}}([0,\infty); D_0^1),\quad (\sqrt\rho u_t,\sqrt\rho\theta_t)\in L^\infty_{\text{loc}}([0,\infty); L^2),
\end{eqnarray*}
provided
$$
\mathscr N_0:=\bar\rho(\|\rho_0\|_3+\bar\rho^2\|\sqrt{\rho_0}u_0\|_2^2)(\|\nabla u_0\|_2^2+
  \bar\rho\|\sqrt{\rho_0}E_0\|_2^2)\leq\varepsilon_0.
$$
\end{theorem}

\begin{remark}
(i) One can easily check that the quantity $\mathscr N_0$ in Theorem \ref{thm} is scaling invariant, with respect to this scaling transformation (\ref{SCAL}). Therefore, Theorem \ref{thm} provides the global well-posedness of system
(\ref{EQrho})--(\ref{EQtheta}) under some smallness assumption on a scaling invariant quantity, for the case that the
vacuum is allowed.

(ii) Global well-posedness of strong solutions to the Cauchy problem of system (\ref{EQrho})--(\ref{EQtheta}) in the presence of vacuum
has been proved in \cite{HUANGLI11} and \cite{WENZHU17}, with non-vacuum far field and vacuum far field, respectively.
The assumptions concerning the smallness in \cite{HUANGLI11} and \cite{WENZHU17} are imposed as
\begin{eqnarray*}
C_0&=&\int \left(\frac{\rho_0}{2}|u_0|^2+R(\rho_0\log\rho_0-\rho_0+1)+\frac{R}{\gamma-1}\rho_0(\theta_0-\log\theta_0+1)\right)dx\\
&\leq&\varepsilon_0=\varepsilon_0(\|\rho_0\|_\infty, \|\theta_0\|_\infty, \|\nabla u_0\|_2, R, \gamma, \mu, \lambda, \kappa)
\end{eqnarray*}
and
$$
\int\rho_0dx\leq\varepsilon_0=\varepsilon_0(\|\rho_0\|_\infty, \|\sqrt{\rho_0}\theta_0\|_2,\|\nabla u_0\|_2,R, \gamma, \mu, \lambda, \kappa ),
$$
respectively. However, since the explicit dependence of $\varepsilon_0$ on $\|\rho_0\|_\infty, \|\theta_0\|_\infty, \|\sqrt{\rho_0}\theta_0\|_2,$ and $\|\nabla u_0\|_2$ are not derived in \cite{HUANGLI11,WENZHU17}, the scaling invariant quantities, on which the smallness guarantees the global well-posedness, can not be identified there.

(iii) Comparing with the global well-posedness result in \cite{WENZHU17}, our result, Theorem \ref{thm}, allows the initial mass to be infinite. This will be crucial for obtaining the global entropy-bounded solutions in our forthcoming paper \cite{LIXIN-3DNON}.
\end{remark}

Comparing with the isentropic case considered in \cite{HLX12}, the additional
difficulty for studying the global well-posedness of
the full compressible Navier-Stokes equations is that the following basic energy inequality does not provide any dissipation estimates:
$$
\int\rho\left(\frac{|u|^2}{2}+c_v\theta\right)dx=\int\rho_0\left(\frac{|u_0|^2}{2}+c_v\theta_0\right)dx.
$$
Note that the dissipation estimates of the form $\int_0^T\|\nabla u\|_2^2\leq C$, which can be guaranteed by the basic energy estimates for the isentropic case, is crucial in
the arguments of \cite{HLX12}. To overcome this difficulty, some kinds of dissipative estimates were recovered
for the full compressible Navier-Stokes equations in \cite{HUANGLI11} and \cite{WENZHU17}, for the cases that with non-vacuum
and vacuum far field, respectively, by using the entropy inequality
and the conservation of mass. Notcing that the entropy inequality, one of the keys in \cite{HUANGLI11},
holds only for the case that with non-vacuum far field, and the finiteness of the mass is required in \cite{WENZHU17},
and recalling that we consider the case that with vacuum far field and allowing possible infinite mass,
the arguments in \cite{HUANGLI11,WENZHU17} do not work for our case.

A crucial ingredient of obtaining the dissipative estimates
is the following new equation (see the proof in Proposition \ref{proprho3})
\begin{equation*}
  \frac{2\mu+\lambda}{2}(\partial_t\rho^3+\text{div}\,(u\rho^3))+\rho^3p+\rho^3\Delta^{-1}\text{div}\,(\rho u)_t+\rho^3\Delta^{-1}\text{div}\,\text{div}\,(\rho u\otimes u)=0,
\end{equation*}
which is derived by combining the continuity equation and the momentum equation; note that the temperature equation plays no role in deriving this. Comparing with the continuity equation, the main advantage of the above equation is that it enables us
to get $L^\infty(0,T; L^3)$ estimate of $\rho$ without appealing to the the $L^\infty$ of $\text{div}\,u$. In fact, the above equation leads to the following kind of inequality
\begin{eqnarray*}
  \sup_{0\leq t\leq T}\|\rho\|_3^3+\int_0^T\int\rho^3pdxdt
  \lesssim \sup_{0\leq t\leq T}(\|\rho\|_\infty^{\frac23}\|\sqrt\rho u\|_2^{\frac13}\|\sqrt\rho|u|^2\|_2^{\frac13}\|\rho\|_3^3)+\cdots,
\end{eqnarray*}
see Proposition \ref{proprho3} for the details. This motivates us to impose the
smallness conditions on $\|\rho_0\|_\infty^2\|\sqrt{\rho_0} u_0\|_2\|\sqrt{\rho_0}|u_0|^2\|_2$ (this is one of
the terms of $\mathscr N_0$ in Theorem \ref{thm}) to get the bound
of $\|\rho\|_3$. The above inequality also guides us to carry out the
estimates on $\|\sqrt\rho u\|_{L^\infty(0,T; L^2)}, \|\sqrt\rho E\|_{L^\infty(0,T; L^2)},$ and $\|\rho\|_{L^\infty(0,T; L^\infty)}$, which are performed in Propositions \ref{propE2.1}, \ref{propE2.2}, and \ref{proprhobdd}, respectively.
Higher order estimates are required in the estimate for $\|\rho\|_{L^\infty(0,T; L^\infty)}$, and they are carried out
with the help of $\omega=\nabla\times u$ and $G=(2\mu+\lambda)\text{div}\,u-p$, which turn out to have better properties than
$\nabla u$, see Proposition \ref{propnablau}. Combining Proposition \ref{propE2.1}, \ref{propE2.2}, \ref{proprho3}, \ref{proprhobdd}, and \ref{propnablau}, by continuity arguments, we are able to get time-independent estimate on the
scaling invariant quantity $\mathscr N_T$, under the condition that $\mathscr N_0$ is sufficiently small.
With this a priori estimate for $\mathscr N_T$, one can further get the time-independent a priori estimates
of $\|\nabla u\|_{L^\infty(0,T; L^2)}$ and $\|\rho\|_{L^\infty(0,T; L^\infty)}$, based on which, the blow-up criteria
apply, and, thus, the global well-posedness follows.

Throughout this paper, we use $C$ to denote a general positive constant which may vary from line to line.
$A\lesssim B$ means $A\leq CB$ for some positive constant $C$.

%

\section{A priori estimates}

This section is devoted to deriving some a priori estimates for the solutions to the Cauchy problem of system (\ref{EQrho})--(\ref{EQtheta}).
The existence of solution is guaranteed by the following local well-posedness result proved in \cite{CHOKIM06-2}:

\begin{proposition}
  \label{LOCAL}
Under the conditions in Theorem \ref{thm}, there is a positive time $T_*$, such that system
(\ref{EQrho})--(\ref{EQtheta}), with initial data $(\rho_0, u_0, \theta_0)$, has a unique solution $(\rho, u, \theta)$, on $\mathbb R^3\times(0,T_*)$, satisfying
\begin{eqnarray*}
  &\rho\in C([0,T_*]; H^1\cap W^{1,q}),\quad\rho_t\in C([0,T_*] L^2\cap L^q),\\
  &(u,\theta)\in C([0,T_*]; D_0^1\cap D^2)\cap L^2(0,T_*; D^{2,q}),\\
  &(u_t,\theta_t)\in L^2(0,T_*; D_0^1),\quad (\sqrt\rho u_t,\sqrt\rho\theta_t)\in L^\infty_{\text{loc}}(0,T_*; L^2).
\end{eqnarray*}
\end{proposition}

In the rest of this section, we always assume that $(\rho, u, \theta)$, is a solution to system (\ref{EQrho})--(\ref{EQtheta}), on $\mathbb R^3\times(0,T),$ for some positive time $T$, satisfying the regularities in Proposition \ref{LOCAL} with $T_*$ there
replaced by $T$, with initial data $(\rho_0, u_0, \theta_0)$.

\subsection{Energy inequalities}

\begin{proposition}
\label{propE2.1}
The following estimate holds:
  $$
  \sup_{0\leq t\leq T}\|\sqrt\rho u\|_2^2+\int_0^T\|\nabla u\|_2^2dt\leq C\|\sqrt{\rho_0}u_0\|_2^2+C\int_0^T\|\rho\|_3^2\|\nabla\theta\|_2^2dt,
  $$
for a positive constant $C$ depending only on $R,\gamma,\mu,\lambda,$ and $\kappa$.
\end{proposition}

\begin{proof}
  Multiplying (\ref{EQu}) by $u$, integration the resultant over $\mathbb R^3$, and noticing that $\mu+\lambda>0$,
  it follows from integration by parts and the Cauchy inequality that
  \begin{eqnarray*}
      &&\frac{d}{dt}\|\sqrt\rho u\|_2^2+\mu\|\nabla u\|_2^2+(\mu+\lambda)\|\text{div}\,u\|_2^2\\
      &\leq& R \|\rho\|_3\|\theta\|_6\|\text{div}u\|_2\leq C\|\rho\|_3\|\nabla\theta\|_2\|\text{div}\,u\|_2\\
      &\leq&(\mu+\lambda)\|\text{div}\,u\|_2^2+C\|\rho\|_3^2\|\nabla\theta\|_2^2,
  \end{eqnarray*}
  from which, the conclusion follows by integrating in $t$.
\end{proof}

\begin{proposition}
  \label{propE2.2}
Assume that $2\mu>\lambda$. Then, the following estimate holds:
\begin{eqnarray*}
  &&\sup_{0\leq t\leq T}\|\sqrt\rho E\|_2^2+\int_0^T(\|\nabla\theta\|_2^2+\||u|\nabla u\|_2^2)dt\\
  &\leq& C\|\sqrt{\rho_0}E_0\|_2^2+C\int_0^T\|\rho\|_\infty\|\rho\|_3^{\frac12}\|\sqrt\rho\theta\|_2\|(\nabla\theta,|u|\nabla u)\|_2^2dt,
  \end{eqnarray*}
  for a positive constant $C$ depending only on $R,\gamma,\mu,\lambda,$ and $\kappa$, where $E=\frac{|u|^2}{2}+c_v\theta$.
\end{proposition}

\begin{proof}
One can verify
\begin{equation}
  \label{EQE}
  \rho(\partial_tE+u\cdot\nabla E)+\text{div}\,(up)-\kappa\Delta\theta=\text{div}\,(\mathcal S\cdot u),
\end{equation}
where $\mathcal S=\mu(\nabla u+(\nabla u)^t)+\lambda\text{div}\,uI$.
Multiplying (\ref{EQE}) by $E$, integrating the resultant over $\mathbb R^3$, it follows from integration by parts that
  \begin{eqnarray*}
    &&\frac12\frac{d}{dt}\|\sqrt\rho E\|_2^2+\kappa c_v\|\nabla\theta\|_2^2\\
    &=&\int\left[-\frac\kappa2\nabla\theta\cdot\nabla|u|^2+(up-\mathcal S\cdot u)\cdot\left(c_v\nabla\theta+\frac{\nabla |u|^2}{2}\right)\right]dx\\
    &\leq&\frac{c_v\kappa}{2}\|\nabla\theta\|_2^2+C\||u|\nabla u\|_2^2+C\int\rho^2\theta^2|u|^2dx,
  \end{eqnarray*}
  which yields
  \begin{equation}
     \frac{d}{dt}\|\sqrt\rho E\|_2^2+\kappa c_v\|\nabla\theta\|_2^2\lesssim\||u|\nabla u\|_2^2+ \int\rho^2\theta^2|u|^2dx.
     \label{E2-1}
  \end{equation}

  Multiplying (\ref{EQu}) by $|u|^2u$, integrating the resultant over $\mathbb R^3$, it follows from integration by parts that
  \begin{eqnarray*}
    \frac14\frac{d}{dt}\|\sqrt\rho|u|^2\|_2^2-\int \mathbb(\mu\Delta u+(\mu+\lambda)\nabla\text{div}\,u)\cdot|u|^2 udx=-\int
    p\text{div}\,(|u|^2u)dx\\
    \leq\,\,\, \left(\mu-\frac\lambda2\right)\int|u|^2|\nabla u|^2 dx+C\int \rho^2\theta^2|u|^2dx,
  \end{eqnarray*}
  Some elementary calculations show that
  $$
  -\int \mathbb(\mu\Delta u+(\mu+\lambda)\nabla\text{div}\,u)\cdot|u|^2 udx\geq (2\mu-\lambda)\int|u|^2|\nabla u|^2dx.
  $$
  Combining the above two inequalities leads to
  \begin{equation}
    \frac{d}{dt}\|\sqrt\rho|u|^2\|_2^2+2(2\mu-\lambda)\||u|\nabla u\|_2^3\lesssim\int\rho^2\theta^2|u|^2dx. \label{E2-2}
  \end{equation}

  Multiplying (\ref{E2-2}) by a sufficient large number $K$ depending only on $R, \gamma, \mu,\lambda,$ and $\kappa$, and summing the resultant
  with (\ref{E2-1}), one obtains
  \begin{equation*}
    \frac{d}{dt}(\|\sqrt\rho E\|_2^2+K\|\sqrt\rho|u|^2\|_2^2)+\kappa c_v\|\nabla\theta\|_2^2+(2\mu-\lambda)K\||u|\nabla u\|_2^2
    \lesssim\int\rho^2\theta^2|u|^2dx,
  \end{equation*}
  from which, noticing that the H\"older and Sobolev inequalities yield
  \begin{eqnarray}
    \int\rho^2\theta^2|u|^2dx&\leq&\|\sqrt\rho\theta\|_2\|\theta\|_6\||u|^2\|_6\|\rho\|_9^{\frac32}\nonumber\\
    &\lesssim&\|\sqrt\rho\theta\|_2\|\nabla\theta\|_2\|\nabla|u|^2\|_2\|\rho\|_\infty\|\rho\|_3^{\frac12},\label{E2-2-1}
  \end{eqnarray}
  one obtains
  \begin{eqnarray*}
   \frac{d}{dt}(\|\sqrt\rho E\|_2^2+K\|\sqrt\rho|u|^2\|_2^2)+\kappa c_v\|\nabla\theta\|_2^2+(2\mu-\lambda)K\||u|\nabla u\|_2^2\\
    \lesssim\|\rho\|_\infty\|\rho\|_3^{\frac12}\|\sqrt\rho\theta\|_2\|\nabla\theta\|_2\|\nabla|u|^2\|_2.
  \end{eqnarray*}
  Integrating this in $t$ and using the Cauchy inequality, the conclusion follows.
\end{proof}

The following proposition on the $L^\infty(0,T; L^3)$ estimate for $\rho$ is crucial in the proof of this paper.

\begin{proposition}
\label{proprho3}
The following estimate holds
\begin{eqnarray*}
  \sup_{0\leq t\leq T}\|\rho\|_3^3+\int_0^T\int\rho^3pdxdt
  &\leq&C\sup_{0\leq t\leq T}(\|\rho\|_\infty^{\frac23}\|\sqrt\rho u\|_2^{\frac13}\|\sqrt\rho|u|^2\|_2^{\frac13}\|\rho\|_3^3)\\
  &&+C\int_0^T\|\rho\|_\infty^2\|\rho\|_3^2\|\nabla u\|_2^2dt+C\|\rho_0\|_3^3,
\end{eqnarray*}
for a positive constant $C$ depending only on $R,\gamma,\mu,\lambda,$ and $\kappa$.
\end{proposition}

\begin{proof}
Applying the operator $\Delta^{-1}\text{div}$ to (\ref{EQu}) yields
\begin{equation}\label{eq}
\Delta^{-1}\text{div}\,(\rho u)_t+\Delta^{-1}\text{div}\,\text{div}\,(\rho u\otimes u)-(2\mu+\lambda)\text{div}\,u+p=0.
\end{equation}
Multiplying the above equation by $\rho^3$ and noticing that
$$
\partial_t \rho ^3+\text{div}\,(u\rho^3)+2\text{div}\,u\rho^3=0,
$$
one obtains
\begin{equation}
  \frac{2\mu+\lambda}{2}(\partial_t\rho^3+\text{div}\,(u\rho^3))+\rho^3p+\rho^3\Delta^{-1}\text{div}\,(\rho u)_t+\rho^3\Delta^{-1}\text{div}\,\text{div}\,(\rho u\otimes u)=0.\label{EQrho3}
\end{equation}
Integrating the above equation over $\mathbb R^3$ yields
\begin{eqnarray}
  \frac{2\mu+\lambda}{2}\frac{d}{dt}\|\rho\|_3^3+\int\rho^3pdx+\int\rho^3
  \Delta^{-1}\text{div}\,(\rho u)_tdx\nonumber\\
  =-\int\rho^3\Delta^{-1}\text{div}\,\text{div}\,(\rho u\otimes u)dx.\label{rho3-1}
\end{eqnarray}
Using (\ref{EQrho}), one deduces
\begin{eqnarray*}
  &&\int\rho^3
  \Delta^{-1}\text{div}\,(\rho u)_tdx
  \\
  &=&\frac{d}{dt}\int\rho^3\Delta^{-1}\text{div}\,(\rho u)dx+\int[\text{div}\,(\rho^3 u)+2\text{div}\,u\rho^3]
  \Delta^{-1}\text{div}\,(\rho u)dx\\
  &=&\int
  [2\text{div}\,u\rho^3 \Delta^{-1}\text{div}\,(\rho u)-\rho^3u\cdot\nabla\Delta^{-1}\text{div}\,(\rho u)]dx\\
  &&+\frac{d}{dt}\int\rho^3\Delta^{-1}\text{div}\,(\rho u)dx.
\end{eqnarray*}
Therefore, it follows from (\ref{rho3-1}) that
\begin{align}
  &\frac{d}{dt}\int\left(\frac{2\mu+\lambda}{2}+ \Delta^{-1}\text{div}\,(\rho u)\right)\rho^3 dx+\int\rho^3 p dx\nonumber\\
  =&~~~\int\left[\rho^3(u\cdot\nabla\Delta^{-1}\text{div}\,(\rho u)-\Delta^{-1}\text{div}\text{div}\,(\rho u\otimes u))-2\text{div}\,u\rho^3\Delta^{-1}\text{div}\,(\rho u)\right]dx.\label{rho3-2}
\end{align}
Noticing that
\begin{eqnarray*}
\|\nabla\Delta^{-1}\text{div}\,(\rho u)\|_2&\lesssim&\|\rho u\|_2\lesssim\|\rho\|_3\|u\|_6,\\
\|\Delta^{-1}\text{div}\,\text{div}\,(\rho u\otimes u)\|_{\frac32}
&\lesssim&\|\rho|u|^2\|_{\frac32}\lesssim\|\rho\|_3\|u\|_6^2,
\end{eqnarray*}
it follows from the H\"older and Sobolev embedding inequality that
\begin{eqnarray}
  \int\rho^3(u\cdot\nabla\Delta^{-1}\text{div}\,(\rho u)-\Delta^{-1}\text{div}\text{div}\,(\rho u\otimes u))dx\nonumber
\\
\lesssim\|\rho \|_9^3\|\rho\|_3\|u\|_6^2\lesssim\|\rho\|_\infty^2\|\rho\|_3^2\|\nabla u\|_2^2.\label{rho3-3}
\end{eqnarray}
By the Sobolev embedding and elliptic estimates
\begin{eqnarray*}
\|\Delta^{-1}\text{div}\,(\rho u)\|_6&\lesssim&\|\nabla \Delta^{-1}\text{div}\,(\rho u)\|_2
\lesssim\|\rho u\|_2\\
&\lesssim&\|\rho\|_3\|u\|_6\lesssim\|\rho\|_3\|\nabla u\|_2,
\end{eqnarray*}
and, thus, the H\"older inequality yields
\begin{equation}
  \label{rho3-4}
  \left|\int\text{div}\,u\rho^3\Delta^{-1}\text{div}\,(\rho u)dx\right|
  \lesssim\|\text{div}u\|_2\|\rho\|_9^3\|\rho\|_3\|\nabla u\|_2\lesssim\|\rho\|_\infty^2\|\rho\|_3^2\|\nabla u\|_2^2.
\end{equation}
By the Gagliardo-Nirenberg inequality and using the elliptic estimates, it follows
\begin{eqnarray}
  \|\Delta^{-1}\text{div}\,(\rho u)\|_\infty&\lesssim&\|\Delta^{-1}\text{div}\,(\rho u)\|_6^{\frac13}\|\nabla\Delta^{-1}\text{div}\,(\rho u)
  \|_4^{\frac23}\nonumber\\
  &\lesssim&\|\rho u\|_2^{\frac13}\|\rho u\|_4^{\frac23}
  \lesssim \|\rho\|_\infty^{\frac23}\|\sqrt{\rho}u\|_2^{\frac13}\|\sqrt{\rho}|u|^2\|_2^{\frac13}.\label{rho3-5}
\end{eqnarray}
Integrating (\ref{rho3-2}) in $t$, using (\ref{rho3-3})--(\ref{rho3-5}), and by some straightforward calculations, the conclusion follows.
\end{proof}

\begin{proposition}
  \label{propnablau}
  Assume $$\sup_{0\leq t\leq T}\|\rho\|_\infty\leq4\bar\rho.$$
  Then, there is a positive constant $C$ depending only on $R,\gamma,\mu,\lambda$, and $\kappa$, such that
  \begin{eqnarray*}
    &&\sup_{0\leq t\leq T}\|\nabla u\|_2^2+\int_0^T\left\|\left(\sqrt\rho u_t,\frac{\nabla G}{\sqrt{\bar\rho}},\frac{\nabla\omega}{\sqrt{\bar\rho}}\right)\right\|_2^2dt\\
    &\leq& C\|\nabla u_0\|_2^2+C\bar\rho\sup_{0\leq t\leq T}\|\sqrt\rho\theta\|_2^2+C\bar\rho^3\int_0^T\|\nabla u\|_2^4\|(\nabla u,\sqrt{\bar\rho}\sqrt\rho\theta)\|_2^2dt\\
    &&+C\int_0^T (\bar\rho+\bar\rho^2\|\rho\|_3^{\frac12}\|\sqrt\rho\theta\|_2)\|(\nabla\theta,|u|\nabla u)\|_2^2dt,
  \end{eqnarray*}
  where $G=(2\mu+\lambda)\text{div}\,u-p$ and $\omega=\nabla\times u$.
\end{proposition}

\begin{proof}
  Multiplying (\ref{EQu}) by $u_t$, integrating the resultant over $\mathbb R^3$, it follows from integration by parts that
  \begin{eqnarray}
    \frac12\frac{d}{dt}(\mu\|\nabla u\|_2^2+(\mu+\lambda)\|\text{div}\,u\|_2^2)-\int p\text{div}\,u_tdx+\|\sqrt\rho u_t\|_2^2\nonumber\\
    =-\int\rho (u\cdot\nabla)u\cdot u_tdx. \label{nablau-1}
  \end{eqnarray}
  Noticing that $\text{div}\,u=\frac{G+p}{2\mu+\lambda}$, it follows
  \begin{eqnarray}
    &&-\int p\text{div}\,u_tdx=-\frac{d}{dt}\int p\text{div}\,udx+\int p_t\text{div}\,udx\nonumber\\
    &=&-\frac{d}{dt}\int p\text{div}\,udx+\frac{1}{2(2\mu+\lambda)}\frac{d}{dt}\|p\|_2^2+\frac{1}{2\mu+\lambda}\int p_tGdx.\label{nablau-2}
  \end{eqnarray}
  Noticing that (\ref{EQtheta}) implies
  $$
  p_t=(\gamma-1)(\mathcal Q(\nabla u)-p\text{div}\,u+\kappa\Delta\theta)-\text{div}\,(up),
  $$
  and, thus, integration by parts gives
  \begin{eqnarray}
    \int p_tGdx = \int[(\gamma-1)(\mathcal Q(\nabla u)-p\text{div}\,u)G+(up-\kappa(\gamma-1)\nabla\theta)\cdot\nabla G]dx. \label{nablau-3}
  \end{eqnarray}
  Substituting (\ref{nablau-3}) into (\ref{nablau-2}), then the resultant into (\ref{nablau-1}), and noticing that $\|\nabla u\|_2^2
  =\|\omega\|_2^2+\|\text{div}\,u\|_2^2$, by some straightforward calculations, one obtains
  \begin{eqnarray}
    &&\frac12\frac{d}{dt}\left( \mu \|\omega\|_2^2+\frac{\|G\|_2^2}{ 2\mu+\lambda }\right)+\|\sqrt\rho u_t\|_2^2\nonumber \\
    &=&-\int\rho (u\cdot\nabla)u\cdot u_tdx+\frac{1}{2\mu+\lambda}\int(\kappa(\gamma-1)\nabla\theta-up)\cdot\nabla G dx\nonumber\\
    &&-\frac{\gamma-1}{2\mu+\lambda}\int (\mathcal Q(\nabla u)-p\text{div}\,u)Gdx. \label{nablau-4}
  \end{eqnarray}

  Use $\Delta u=\nabla\text{div}\,u-\nabla\times\nabla\times u$ to rewrite (\ref{EQu}) as
  \begin{equation}\label{EQU'}
  \rho(u_t+u\cdot\nabla u)=\nabla G-\mu\nabla\times\omega.
  \end{equation}
  Testing this by $\nabla G$, noticing $\int\nabla G\cdot\nabla\times\omega dx=0$, and recalling $\|\rho\|_\infty\leq4\bar\rho$ yield
  \begin{eqnarray*}
\|\nabla G\|_2^2&=&\int\rho(u_t+u\cdot\nabla u)\cdot\nabla Gdx\\
    &\leq&\int\left(\frac{|\nabla G|^2}{2} +2\bar\rho\rho|u_t|^2\right)dx+\int\rho u\cdot\nabla u\cdot\nabla Gdx,
  \end{eqnarray*}
  which gives
  \begin{equation}
    \frac{\|\nabla G\|_2^2}{16\bar\rho}\leq \frac14\|\sqrt\rho u_t\|_2^2+\frac{1}{8\bar\rho}\int\rho (u\cdot\nabla) u\cdot\nabla Gdx.
    \label{nGellpticEst}
  \end{equation}
  Similarly
  \begin{equation}
    \frac{\mu^2\|\nabla\omega\|_2^2}{16\bar\rho}\leq\frac14\|\sqrt\rho u_t\|_2^2+\frac{1}{8\bar\rho}\int\rho(u\cdot\nabla)u\cdot\nabla\times\omega dx.
    \label{nomegaellipticEst}
  \end{equation}
  Thanks to (\ref{nGellpticEst}) and (\ref{nomegaellipticEst}), one obtains from
  (\ref{nablau-4}) that
  \begin{align}
    &\frac12\frac{d}{dt}\left( \mu \|\omega\|_2^2+\frac{\|G\|_2^2}{ 2\mu+\lambda }\right)+\frac12\|\sqrt\rho u_t\|_2^2
    +\frac{1}{16\bar\rho}(\|\nabla G\|_2^2+\mu^2\|\nabla\omega\|_2^2)\nonumber\\
    \leq& C\int\rho|u||\nabla u|\left[|u_t|+\frac{1}{\bar\rho}(\nabla G|+|\nabla\omega|)\right]dx+C\int(|\nabla\theta|+\rho\theta|u|)|\nabla G|dx\nonumber\\
    &+C\int(|\nabla u|^2+\rho\theta|\nabla u|)|G|dx
    =:I_1+I_2+I_3.\label{nablau-5}
  \end{align}

  The terms $I_1, I_2,$ and $I_3$ are estimated as follows. For $I_1$, by the H\"older and Young inequalities, one obtains
  \begin{eqnarray*}
    I_1&\lesssim&\sqrt{\bar\rho}\||u|\nabla u\|_2\|\sqrt\rho u_t\|_2+\||u|\nabla u\|_2(\|\nabla G\|_2+\|\nabla\omega\|_2)\\
    &\leq&\frac16\left[\frac12\|\sqrt\rho u_t\|_2^2
    +\frac{1}{16\bar\rho}(\|\nabla G\|_2^2+\mu^2\|\nabla\omega\|_2^2)
    \right]+C\bar\rho\||u|\nabla u\|_2^2.
  \end{eqnarray*}
  Recalling (\ref{E2-2-1}), it follows from the H\"older and Young inequalities that
  \begin{eqnarray*}
    I_2&\lesssim&\|\nabla\theta\|_2\|\nabla G\|_2+\|\rho\theta u\|_2\|\nabla G\|_2\\
    &\lesssim&\|\nabla\theta\|_2\|\nabla G\|_2+\sqrt{\bar\rho}\|\rho\|_3^{\frac14}\|\sqrt\rho\theta\|_2^{\frac12}\|\nabla\theta\|_2^{\frac12}
    \|\nabla|u|^2\|_2^{\frac12}\|\nabla G\|_2\\
     &\leq&\frac{\|\nabla G\|_2^2}{96\bar\rho}+C\bar\rho^2\|\rho\|_2^{\frac12}\|\sqrt\rho\theta\|_2(\|\nabla\theta\|_2^2
     +\||u|\nabla u\|_2^2).
  \end{eqnarray*}
  The elliptic estimates and Sobolev embedding inequality yield
  \begin{eqnarray}
    \|\nabla u\|_6&\lesssim&\|\nabla\times u\|_6+\|\text{div}\,u\|_6\lesssim\|\omega\|_6+\|G\|_6+\|\rho\theta\|_6\nonumber\\
    &\lesssim&\|\nabla\omega\|_2+\|\nabla G\|_2+{\bar\rho}\|\nabla\theta\|_2.\label{nu6}
  \end{eqnarray}
  Using (\ref{nu6}), by the H\"older, Sobolev, and Young inequalities, one deduces
  \begin{eqnarray*}
    I_3&\lesssim&\|\nabla u\|_2\|\nabla u\|_6\|G\|_3+\|\nabla u\|_2\|\rho\theta\|_6\|G\|_3\nonumber\\
    &\lesssim& C\|\nabla u\|_2(\|\nabla G\|_2+\|\nabla\omega\|_2+\bar\rho\|\nabla\theta\|_2)\|G\|_2^{\frac12}\\
    &&+\|\nabla G\|_2^{\frac12}\bar\rho\|\nabla u\|_2\|\nabla\theta\|_2\|G\|_2^{\frac12}\|\nabla G\|_2^{\frac12}\\
    &\leq&\frac{1}{96\bar\rho}(\|\nabla G\|_2^2+\mu^2\|\nabla\omega\|_2^2)
    +C\bar\rho^3\|\nabla u\|_2^4\|G\|_2^2\\
    &&+C\bar\rho(\|\nabla\theta\|_2^2+\||u|\nabla u\|_2^2).
  \end{eqnarray*}

  Substituting the estimates for $I_i, i=1,2,3,$ into (\ref{nablau-5}) yields
  \begin{eqnarray}
    \frac{d}{dt}\left( \mu \|\omega\|_2^2+\frac{\|G\|_2^2}{ 2\mu+\lambda }\right)+\frac12\|\sqrt\rho u_t\|_2^2
    +\frac{1}{16\bar\rho}(\|\nabla G\|_2^2+\mu^2\|\nabla\omega\|_2^2)\nonumber\\
    \lesssim (\bar\rho+\bar\rho^2\|\rho\|_3^{\frac12}\|\sqrt\rho\theta\|_2)(\|\nabla\theta\|_2^2+
    \||u|\nabla u\|_2^2)+\bar\rho^3\|\nabla u\|_2^4\|G\|_2^2, \nonumber
  \end{eqnarray}
  from which, integrating in $t$ and using
  \begin{equation*}
  \|\nabla u\|_2\lesssim\|\omega\|_2+\|G\|_2+\|\rho\theta\|_2\lesssim\|\omega\|_2+\|G\|_2+\sqrt{\bar\rho}\|\sqrt\rho\theta\|_2,
  \end{equation*}
  the conclusion follows by straightforward calculations.
\end{proof}

\begin{proposition}
  \label{proprhobdd}
  Assume $$\sup_{0\leq t\leq T}\|\rho\|_\infty\leq4\bar\rho.$$ Then, there is a positive constant $C$ depending only on $R,\gamma,\mu,\lambda,$ and $\kappa$, such that
  \begin{eqnarray*}
    \sup_{0\leq t\leq T}\|\rho\|_\infty&\leq&\|\rho_0\|_\infty e^{C\bar\rho^{\frac23}\sup_{0\leq t\leq T}\|\sqrt\rho u\|_2^{\frac13}\|\sqrt\rho|u|^2\|_2^{\frac13}+C\bar\rho\int_0^T\|\nabla u\|_2\|(\nabla G,\nabla\omega,\sqrt{\bar\rho}\nabla\theta)\|_2dt}.
  \end{eqnarray*}
\end{proposition}

\begin{proof}
  Denote $\mathcal O=\{x\in\mathbb R^3|\rho_0(x)=0\}$ and $\Omega=\{x\in\mathbb R^3|\rho_0(x)>0\}$. Let $X(x,t)$ be the particle path starting from $x$ and govern by the velocity field $u$, that is
  $$
  \partial_tX(x,t)=u(X(x,t),t),\quad X(x,0)=x.
  $$
  Then $\rho(X(x,t),t)\equiv0$, for any $x\in\mathcal O$, and $\rho(X(x,t),t)>0$, for any $x\in\Omega$.
  One can verify that $\{X(x,t)|x\in\mathbb R^3\}=\mathbb R^3$, for any $t\in(0,T)$. Therefore
  \begin{equation}
  \sup_{x\in\mathbb R^3}\rho(x,t)=\sup_{x\in\mathbb R^3}\|\rho(X(x,t),t)\|_\infty=\sup_{x\in\Omega}\rho(X(x,t),t).\label{ubdid}
  \end{equation}
  Rewrite (\ref{eq}) as
  \begin{eqnarray}
    \label{eq'}
    &&\partial_t\Delta^{-1}\text{div}\,(\rho u)+u\cdot\nabla\Delta^{-1}\text{div}\,\text{div}\,(\rho u)-(2\mu+\lambda)\text{div}\,u+p \nonumber\\
    &=&u\cdot\nabla\Delta^{-1}\text{div}\,\text{div}\,(\rho u)-\Delta^{-1}\text{div}\,\text{div}\,(\rho u\otimes u)
    =[u,\mathcal R\otimes\mathcal R](\rho u),
  \end{eqnarray}
  where $\mathcal R$ is the Riesz transform on $\mathbb R^3$.
  Using the fact $\frac{d}{dt}(f(X(x,t),t))=(\partial_t f+u\cdot\nabla f)(X(x,t),t)$, it follows from (\ref{EQrho}) that
  $$
  \frac{d}{dt}(\log\rho(X(x,t),t))=-\text{div}\,u(X(x,t),t), \quad\forall x\in\Omega.
  $$
  Therefore, for any $x\in\Omega,$ it follows from (\ref{eq'}) that
  \begin{align*}
  \frac{d}{dt}\Big((2\mu+\lambda)\log\rho(X(x,t),t)+(\Delta^{-1}\text{div}\,(\rho u))(X(x,t),t)\Big)\nonumber\\
  +p(X(x,t),t)
  =\Big([u,\mathcal R\otimes\mathcal R](\rho u)\Big)
  (X(x,t),t). \nonumber
  \end{align*}
  Due to $p\geq0$ and (\ref{ubdid}), one can easily derive from the above equality that
  \begin{equation}
    \|\rho\|_\infty\leq\|\rho_0\|_\infty e^{C \left( \sup_{0\leq t\leq T}\|\Delta^{-1}\text{div}\,(\rho u)\|_\infty+\int_0^T\|[u,\mathcal R\otimes\mathcal R](\rho u)\|_\infty dt\right)}.\label{ubdrho-1}
  \end{equation}
  Using the Gagliardo-Nirenberg inequality and the commutator estimates, one deduces
  \begin{align*}
    \|[u,\mathcal R\otimes\mathcal R]&(\rho u)\|_\infty\lesssim\|[u,\mathcal R\otimes\mathcal R](\rho u)\|_3^{\frac15}\|\nabla
    [u,\mathcal R\otimes\mathcal R](\rho u)\|_4^{\frac45}\\
    \lesssim&\,\,\,\,\|u\|_6^{\frac15}\|\rho u\|_6^{\frac15}\|\nabla u\|_6^{\frac45}\|\rho u\|_{12}^{\frac45}
    \lesssim\bar\rho\|u\|_6^{\frac15}\|u\|_6^{\frac15}\|\nabla u\|_6^{\frac45}\left(\|u\|_6^{\frac34}\|\nabla u\|_6^{\frac14}\right)^{\frac45}\\
    \lesssim&\,\,\,\,\bar\rho\|\nabla u\|_2\|\nabla u\|_6\lesssim\bar\rho\|\nabla u\|_2(\|\nabla G\|_2+\|\nabla\omega\|_2+\bar\rho\|\nabla\theta\|_2),
  \end{align*}
  where, in the last step, (\ref{nu6}) has been used. Thanks to this and recalling (\ref{rho3-5}), the conclusion follows from (\ref{ubdrho-1}).
\end{proof}

\subsection{A priori estimates}
\begin{proposition}
\label{propkey}
Assume that $2\mu>\lambda$. Denote
$$
\mathscr N_T=\bar\rho\sup_{0\leq t\leq T}(\|\rho\|_3+\bar\rho^2\|\sqrt{\rho}u\|_2^2)(t)\sup_{0\leq t\leq T}(\|\nabla u\|_2^2+
  \bar\rho\|\sqrt{\rho}E\|_2^2)(t).
$$
Then, there is a positive constant $\eta_0$ depending only on $R,\gamma,\mu,\lambda,$ and $\kappa$, such that if
$$
\eta\leq\eta_0,\quad \sup_{0\leq t\leq T}\|\rho\|_\infty\leq 4\bar\rho,\quad\mbox{and}\quad\mathscr N_T\leq\sqrt\eta,
$$
then the following estimates hold
\begin{eqnarray*}
\sup_{0\leq t\leq T}\|\sqrt\rho E\|_2^2+\int_0^T\|(\nabla\theta,|u|\nabla u)\|_2^2dt&\leq& C\|\sqrt{\rho_0}E_0\|_2^2,\\
\sup_{0\leq t\leq T}\|\rho\|_3+\left(\int_0^T\int\rho^3pdxdt\right)^{\frac13}
&\leq& C(\|\rho_0\|_3+\bar\rho^2\|\sqrt{\rho_0}u_0\|_2^2),\\
\bar\rho^2\left(\sup_{0\leq t\leq T}\|\sqrt\rho u\|_2^2
  + \int_0^T\|\nabla u\|_2^2dt \right)
  &\leq&C(\|\rho_0\|_3+\bar\rho^2\|\sqrt{\rho_0}u_0\|_2^2),\\
\sup_{0\leq t\leq T}\|\nabla u\|_2^2+\int_0^T\left\|\left(\sqrt\rho u_t,\frac{\nabla G}{\sqrt{\bar\rho}},\frac{\nabla\omega}
{\sqrt{\bar\rho}}\right)\right\|_2^2dt&\leq& C(\|\nabla u_0\|_2^2+\bar\rho\|\sqrt{\rho_0}E_0\|_2^2), \\
\sup_{0\leq t\leq T}\|\rho\|_\infty &\leq& \bar\rho e^{C\mathscr N_0^{\frac16}+C\mathscr N_0^{\frac12}},
\end{eqnarray*}
for a positive constant $C$ depending only on $R,\gamma,\mu,\lambda,$ and $\kappa$,
where
$$
\mathscr N_0=\bar\rho(\|\rho_0\|_3+\bar\rho^2\|\sqrt{\rho_0}u_0\|_2^2)(\|\nabla u_0\|_2^2+
  \bar\rho\|\sqrt{\rho_0}E_0\|_2^2).
$$
\end{proposition}

\begin{proof}
By assumptions, it follows from Proposition \ref{propE2.2} that
\begin{eqnarray*}
  &&\sup_{0\leq t\leq T}\|\sqrt\rho E\|_2^2+\int_0^T(\|\nabla\theta\|_2^2+\||u|\nabla u\|_2^2)dt\nonumber\\
  &\leq&C\|\sqrt{\rho_0}E_0\|_2^2+C\eta_0^{\frac14}\int_0^T(\|\nabla\theta\|_2^2+\||u|\nabla u\|_2^2)dt,
\end{eqnarray*}
which by choosing $\eta_0$ suitably small implies
\begin{equation}
  \label{EST1}
  \sup_{0\leq t\leq T}\|\sqrt\rho E\|_2^2+\int_0^T(\|\nabla\theta\|_2^2+\||u|\nabla u\|_2^2)dt\leq C\|\sqrt{\rho_0}E_0\|_2^2.
\end{equation}
Thanks to (\ref{EST1}) and applying Proposition \ref{propE2.1}, one obtains
\begin{eqnarray}
   \sup_{0\leq t\leq T}\|\sqrt\rho u\|_2^2+\int_0^T\|\nabla u\|_2^2dt
  &\leq&C\|\sqrt{\rho_0}u_0\|_2^2+C\|\sqrt{\rho_0}E_0\|_2^2\sup_{0\leq t
  \leq T}\|\rho\|_3^2\nonumber\\
  &\leq&C\|\sqrt{\rho_0}u_0\|_2^2+C\sup_{0\leq t
  \leq T}\|\sqrt{\rho}E\|_2^2\sup_{0\leq t
  \leq T}\|\rho\|_3^2\nonumber\\
  &\leq&C\|\sqrt{\rho_0}u_0\|_2^2+\frac{C\sqrt{\eta_0}}{\bar\rho^2}\sup_{0\leq t
  \leq T}\|\rho\|_3. \label{EST1'}
\end{eqnarray}
Using the assumptions and (\ref{EST1'}), it follows from Proposition \ref{proprho3} and the Young inequality that
\begin{eqnarray*}
  &&\sup_{0\leq t\leq T}\|\rho\|_3^3+\int_0^T\int\rho^3pdxdt\\
  &\leq&C\|\rho_0\|_3^3+C\eta_0^{\frac{1}{12}}\sup_{0\leq t\leq T}\|\rho_0\|_3^3+C\bar\rho^2\left(\|\sqrt{\rho_0}u_0\|_2^2+\frac{\sqrt{\eta_0}}{\bar\rho^2}\sup_{0\leq t
  \leq T}\|\rho\|_3\right)\sup_{0\leq
  t\leq T}\|\rho\|_3^2\nonumber\\
  &\leq&C\|\rho_0\|_3^3+\left(C\eta_0^{\frac{1}{12}}+\frac14+C\sqrt{\eta_0}\right)\sup_{0\leq t\leq T}\|\rho\|_3^3
  +C\bar\rho^6\|\sqrt{\rho_0}u_0\|_2^6,
\end{eqnarray*}
from which, by choosing $\eta_0$ sufficiently small, one obtains
\begin{equation}
  \label{EST2}
  \sup_{0\leq t\leq T}\|\rho\|_3+\left(\int_0^T\int\rho^3pdxdt\right)^{\frac13}
  \leq  C(\|\rho_0\|_3+\bar\rho^2\|\sqrt{\rho_0}u_0\|_2^2).
\end{equation}
Combing (\ref{EST1'}) with (\ref{EST2}) yields
\begin{equation}
  \label{EST3}
  \bar\rho^2\left(\sup_{0\leq t\leq T}\|\sqrt\rho u\|_2^2
  + \int_0^T\|\nabla u\|_2^2dt \right)
  \leq C(\|\rho_0\|_3+\bar\rho^2\|\sqrt{\rho_0}u_0\|_2^2).
\end{equation}
Using (\ref{EST1}) and (\ref{EST3}), it follows from Proposition \ref{propnablau} that
\begin{eqnarray}
  &&\sup_{0\leq t\leq T}\|\nabla u\|_2^2+\int_0^T\left\|\left(\sqrt\rho u_t,\frac{\nabla G}{\sqrt{\bar\rho}},\frac{\nabla\omega}
{\sqrt{\bar\rho}}\right)\right\|_2^2dt\nonumber\\
&\lesssim&\|\nabla u_0\|_2^2+\bar\rho\|\sqrt{\rho_0}E_0\|_2^2+\bar\rho^3\int_0^T\|\nabla u\|_2^2dt
\sup_{0\leq t\leq T}\left(\|\nabla u\|_2^2+\bar\rho\|\sqrt\rho\theta\|_2^2\right)\nonumber\\
&&\times\sup_{0\leq t\leq T}\|\nabla u\|_2^2+ \left(\bar\rho+\bar\rho^2\sup_{0\leq t\leq T}\|\rho\|_3^{\frac12}\|\sqrt\rho\theta\|_2\right) \int_0^T\|(\nabla\theta,|u|\nabla u)\|_2^2dt\nonumber\\
&\lesssim&\|\nabla u_0\|_2^2+\bar\rho\|\sqrt{\rho_0}E_0\|_2^2+ \bar\rho(\|\rho_0\|_3+\bar\rho^2\|\sqrt{\rho_0}u_0\|_2^2)
\nonumber\\
&&\times\sup_{0\leq t\leq T}\left(\|\nabla u\|_2^2+\bar\rho\|\sqrt\rho E\|_2^2\right)\sup_{0\leq t\leq T}\|\nabla u\|_2^2\nonumber\\
&&+ \bar\rho^2\sup_{0\leq t\leq T}\|\rho\|_3^{\frac12}\|\sqrt\rho\theta\|_2\|\sqrt{\rho_0}E_0\|_2^2.
\label{EST3-1}
  \end{eqnarray}
Recalling the definition of $\mathscr N_T$ and the assumption that $\mathscr N_T\leq\sqrt{\eta_0}$, it is clear that
\begin{eqnarray*}
  &&\bar\rho(\|\rho_0\|_3+\bar\rho^2\|\sqrt{\rho_0}u_0\|_2^2)
\sup_{0\leq t\leq T}\left(\|\nabla u\|_2^2+\bar\rho\|\sqrt\rho E\|_2^2\right) \\
&\leq&\bar\rho\sup_{0\leq t\leq T}(\|\rho\|_3+\bar\rho^2\|\sqrt{\rho}u\|_2^2)
\sup_{0\leq t\leq T}\left(\|\nabla u\|_2^2+\bar\rho\|\sqrt\rho E\|_2^2\right) \leq \mathscr N_T\leq\sqrt{\eta_0}
\end{eqnarray*}
and
\begin{eqnarray*}
  \bar\rho \sup_{0\leq t\leq T}\|\rho\|_3^{\frac12}\|\sqrt\rho\theta\|_2\leq\left(\bar\rho^2\sup_{0\leq t\leq T}\|\rho\|_3\sup_{0\leq t
  \leq T}\|\sqrt\rho E\|_2^2\right)^{\frac12}\leq\mathscr N_T^{\frac12}\leq\eta_0^{\frac14}.
\end{eqnarray*}
Thanks to the above two estimates, by choosing $\eta_0$ sufficiently small, one can easily derive from (\ref{EST3-1}) that
\begin{equation}
  \sup_{0\leq t\leq T}\|\nabla u\|_2^2+\int_0^T\left\|\left(\sqrt\rho u_t,\frac{\nabla G}{\sqrt{\bar\rho}},\frac{\nabla\omega}
{\sqrt{\bar\rho}}\right)\right\|_2^2dt \leq C(\|\nabla u_0\|_2^2+\bar\rho\|\sqrt{\rho_0}E_0\|_2^2). \label{EST4}
\end{equation}
The estimate for $\|\rho\|_\infty$ follows from Proposition \ref{proprhobdd} by using (\ref{EST1}), (\ref{EST3}), and (\ref{EST4}).
\end{proof}

\begin{proposition}
\label{bridge}
Assume that $2\mu>\lambda$. Let $\eta_0$, $\mathscr N_T$, and $\mathscr N_0$ be as in Proposition \ref{propkey}. Then, the following two hold:

(i) There is a number $\varepsilon_0\in(0,\eta_0)$ depending only on $R, \gamma, \mu, \lambda,$ and $\kappa,$ such that
if
$$
\sup_{0\leq t\leq T}\|\rho\|_\infty\leq 4\bar\rho,\quad  \mathscr N_T\leq\sqrt{\varepsilon_0},\quad and \quad \mathscr N_0\leq\varepsilon_0.
$$
then
$$
\sup_{0\leq t\leq T}\|\rho\|_\infty\leq 2\bar\rho\quad\mbox{and}\quad \mathscr N_T\leq\frac{\sqrt{\varepsilon_0}}{2}.
$$

(ii) As a consequence of (i), the following estimates hold
$$
\mathscr N_{T}\leq\frac{\sqrt{\varepsilon_0}}{2} \quad\mbox{and}\quad \sup_{0\leq t\leq T}\|\rho\|_\infty\leq 2\bar\rho,
$$
as long as $\mathscr N_0\leq\varepsilon_0$.
\end{proposition}

\begin{proof}
(i) Let $\varepsilon_0\leq\eta_0$ be sufficiently small. By assumptions, all the conditions in Proposition \ref{propkey} hold, and, thus
\begin{eqnarray*}
\mathscr N_T&\leq&   C\bar\rho(\|\rho_0\|_3+\bar\rho^2\|\sqrt{\rho_0}u_0\|_2^2)(\|\nabla u_0\|_2^2+
  \bar\rho\|\sqrt{\rho_0}E_0\|_2^2)\\
  &=&C\mathscr N_0\leq C\varepsilon_0\leq\frac{\sqrt{\varepsilon_0}}{2},
\end{eqnarray*}
and
\begin{eqnarray*}
  \sup_{0\leq t\leq T}\|\rho\|_\infty &\leq& \bar\rho e^{C\mathscr N_0^{\frac16}+C\mathscr N_0^{\frac12}}
   \leq  \bar\rho e^{C\varepsilon_0^{\frac16}+C\varepsilon_0^{\frac12}}\leq 2\bar\rho,
\end{eqnarray*}
as long as $\varepsilon_0$ is sufficiently small. The first conclusion follows.

(ii) Define
$$
T_\#:=\max\left\{\mathcal T\in(0,T]\bigg|\mathscr N_\mathcal T\leq\sqrt{\varepsilon_0}, \sup_{0\leq t\leq \mathcal T}\|\rho\|_\infty\leq 4\bar\rho\right\}.
$$
Then, by (i), we have
\begin{equation}\label{6.8-1}
\mathscr N_\mathcal T\leq\frac{\sqrt{\varepsilon_0}}{2}, \qquad \sup_{0\leq t\leq \mathcal T}\|\rho\|_\infty\leq 2\bar\rho,\quad \forall \mathcal T\in(0,T_\#).
\end{equation}
If $T_\#<T$, noticing that $\mathscr N_\mathcal T$ and $\sup_{0\leq t\leq \mathcal T}\|\rho\|_\infty$ are continuous on $[0,T]$,
there is another time $T_{\#\#}\in(T_\#,T]$, such that
$$
\mathscr N_{T_{\#\#}}\leq{\sqrt{\varepsilon_0}} \qquad\mbox{and}\quad \sup_{0\leq t\leq T_{\#\#}}\|\rho\|_\infty\leq 4\bar\rho,
$$
which contradicts to the definition of $T_\#$. Thus, we
have $T_\#=T$, and
the conclusion follows from (\ref{6.8-1}) and the continuity of $\mathscr N_\mathcal T$ and $\sup_{0\leq t\leq \mathcal T}\|\rho\|_\infty$ on $[0,T]$.
\end{proof}

The following corollary is a straightforward consequence of Proposition \ref{propkey} and (ii) of Proposition \ref{bridge}.

\begin{corollary}
\label{CorKey}
Assume that $2\mu>\lambda$. Let $\varepsilon_0$ be as in Proposition \ref{bridge} and assume $\mathscr N_0\leq\varepsilon_0$. Then, there is a positive constant $C$ depending only on $R$, $\gamma$, $\mu$, $\lambda$, $\kappa$, $\bar\rho$, $\|\rho_0\|_3$, $\|\sqrt{\rho_0}u_0\|_2$, $\|\sqrt{\rho_0}E_0\|_2,$ and $\|\nabla u_0\|_2$, such that the following estimates hold:
\begin{eqnarray*}
\sup_{0\leq t\leq T}(\|(\sqrt\rho E,\sqrt\rho u,\nabla u)\|_2^2+\|\rho\|_3+\|\rho\|_\infty) \leq C,\\
\int_0^T\left(\|\left(\nabla\theta,|u|\nabla u, \sqrt\rho u_t, \nabla G, \nabla\omega\right)\|_2^2+\|\nabla u\|_6^2+\int\rho^3pdx\right)dt\leq C.
\end{eqnarray*}
\end{corollary}

\section{Proof of Theorem \ref{thm}}

The following blow-up criteria is cited from Huang--Li \cite{HUANGLI13CMP}.

\begin{proposition}
  \label{CRITERIA}
Let $T^*<\infty$ be the maximal time of existence of
a solution $(\rho, u, \theta)$ to system (\ref{EQrho})--(\ref{EQtheta}), with initial data $(\rho_0, u_0, \theta_0)$. Then,
$$
\lim_{T\rightarrow T^*}(\|\rho\|_{L^\infty(0,T; L^\infty)}+\|u\|_{L^s(0,T;L^r)})=\infty,
$$
for any $(s,r)$ such that $\frac2s+\frac3r\leq1$ and $3<r\leq\infty$.
\end{proposition}

We are now ready to prove Theorem \ref{thm}.

\begin{proof}[\textbf{Proof of Theorem \ref{thm}}]
Let $\varepsilon_0$ and $\mathscr N_T$ be as in Proposition \ref{bridge} and assume $\mathscr N_0\leq\varepsilon_0$.
By Proposition \ref{LOCAL}, there is a unique local strong solution $(\rho, u,\theta)$ to system (\ref{EQrho})--(\ref{EQtheta}),
with initial data $(\rho_0, u_0, \theta_0)$. Extend the local solution $(\rho, u, \theta)$ to the maximal time of existence $T_\text{max}$. If $T_\text{max}=\infty$, then $(\rho, u, \theta)$ is a global solution and we are down. Assume that $T_\text{max}<\infty$.
Then, by the blow up criteria in Proposition \ref{CRITERIA}, it holds
\begin{equation}
  \lim_{T\rightarrow T_\text{max}}(\|\rho\|_{L^\infty(0,T; L^\infty)}+\|u\|_{L^4(0,T;L^6)})=\infty.
  \label{PF1}
\end{equation}
By Corollary \ref{CorKey}, it follows $\sup_{0\leq t\leq T}(\|\rho\|_\infty+\|\nabla u\|_2^2)\leq C$ which, by the Sobolev embedding inequality, gives
\begin{equation*}
\|\rho\|_{L^\infty(0,T; L^\infty)}+\|u\|_{L^4(0,T; L^6)}\leq C,
\end{equation*}
for any $T\in(0,T_\text{max})$, and for a positive constant $C$ independent of $T$. This implies
$$
\lim_{T\rightarrow T_\text{max}}(\|\rho\|_{L^\infty(0,T; L^\infty)}+\|u\|_{L^4(0,T;L^6)})\leq C<\infty,
$$
contradicting to (\ref{PF1}). Therefore, we must have $T_\text{max}=\infty$, proving Theorem \ref{thm}.
\end{proof}

\section*{Acknowledgments}
{J.Li was partly supported by start-up fund 550-8S0315 of the South China Normal University, the NSFC under 11771156 and 11871005, and the Hong Kong RGC Grant
CUHK-14302917.}
\par

\end{document}